\documentclass[12pt]{article}

\usepackage{geometry} 
\usepackage{amsthm,amssymb}
\usepackage{amsmath,amsfonts,latexsym}
\usepackage{latexsym}
\usepackage{float}
\usepackage{xcolor}
\usepackage{url}

\newtheorem{defi}{Definition}

\newtheorem{tm}{Theorem}
\newtheorem{lema}{Lemma}
\newtheorem{kor}{Corollary}

\theoremstyle{remark}
\newtheorem{remark}{Remark}
\newtheorem{ex}[remark]{Example}

\begin{document}


\vspace*{0.2cm}

\begin{center}
{\Large\bf LCD subspace codes}       
\end{center}

\vspace*{0.2cm}

\begin{center}
Dean Crnkovi\'c \\
({\it\small E-mail: deanc@math.uniri.hr})\\[3pt] 
Andrea \v Svob \footnote{Corresponding author}\\
({\it\small E-mail: asvob@math.uniri.hr})\\[3pt]
{\it\small Faculty of Mathematics, University of Rijeka} \\
{\it\small Radmile Matej\v ci\'c 2, 51000 Rijeka, Croatia}\\
\end{center}

\vspace*{0.5cm}

\begin{abstract}
A subspace code is a nonempty set of subspaces of a vector space $\mathbb F^n_q$. Linear codes with complementary duals, or LCD codes, are linear codes whose intersection with their duals is trivial. In this paper, we introduce a notion of LCD subspace codes. 
We show that the minimum distance decoding problem for an LCD subspace code reduces to a problem that is simpler than for a general subspace code.
Further, we show that under some conditions equitable partitions of association schemes yield such LCD subspace codes and as an illustration of the method give some examples from distance-regular graphs.
We also give a construction from mutually unbiased Hadamard matrices, and more generally, from mutually unbiased weighing matrices.
\end{abstract}

\vspace*{0.5cm}




\section{Introduction}\label{intro_new}

A network is a directed multigraph which consists of different vertices.
The source vertices transmit messages to the sink
vertices through a channel of inner vertices. The idea of network coding is to allow mixing of data at these intermediate network vertices. A receiver sees the data packets and deduces from them the messages that were originally intended for the sinks. 
When sending information through a network, we can do linear combinations
on the intermediate nodes and then call this network linear network coding. 

In \cite{network-coding}, R. K\"{o}tter and F. Kschischang opened a new area of research in the coding theory and named it the theory of random  network  coding. It was developed with an idea of the transmission of information in networks. In classical coding theory vectors are sent as codewords and here, in random network coding the codewords are subspaces. This is the reason why these types of codes are called subspace codes. R. K\"{o}tter and F. Kschischang proved (see \cite{network-coding}) that subspace codes are efficient for transmission in networks and because of their applications in error correction for random network coding they attract a wide attention in recent research. See for example \cite{daniele, leo-constant, zhang}.

Linear codes with complementary duals, or LCD codes, are linear codes whose intersection with their duals is trivial. This type of linear codes were introduced by Massey in \cite{massey} and have been widely applied in information protection, electronics and cryptography. It is known that they provide an optimum linear coding solution for the two user binary adder channel and they can be used to protect systems against side channel attacks (SCA) and fault injection attacks (FIA). See for example \cite{multi, carlet-guilley}. LCD codes are asymptotically good; it is shown in \cite{sendrier} that they meet the asymptotic Gilbert-Varshamov bound. Nowadays, many research topics are dealing with LCD codes.  See for example \cite{araya-har, stefka, dra-lcd}.

In this paper, we are interested in LCD subspace codes. We introduce a notion of LCD subspace codes and give a construction of such a codes via equitable partitions of association schemes. Further, we give examples from distance-regular graphs. We also show that the minimum distance decoding problem for an LCD subspace code reduces to a problem that is simpler than for a general subspace code.

The paper is outlined as follows. In the next section, we provide the relevant background information. In Section \ref{LCD-subspace}, we introduce a notion of LCD subspace codes and in Section 
\ref{constr-LCD}, we give a method for obtaining such subspace codes from association schemes and give examples of the construction from distance-regular graphs. 
In Section \ref{MUH}, we give a construction of LCD subspace codes from mutually unbiased Hadamard matrices, and more generally, from mutually unbiased weighing matrices.
In this work, we have used computer algebra systems GAP \cite{GAP2016} and Magma \cite{magma}. 

\section{Preliminaries} \label{intro}

We assume that the reader is familiar with the basic facts of theory of distance-regular graphs and association schemes (for background reading in theory of distance-regular graphs and association schemes we refer the reader to \cite{BCN}). We also assume a basic knowledge of coding theory (see \cite{FEC}). For recent results on subspace codes we refer to \cite{network-book, heinlein-dcc}. 
Throughout, we let $\mathbb{F}_{q}$ denote the finite field of order $q$ for some prime power $q$.

A $q$-ary \emph{linear code} $C$ of dimension $k$ for a prime power $q$, is a $k$-dimensional subspace of a vector space $\mathbb{F}_{q}^n$ over $\mathbb{F}_{q}$.  Elements of $C$ are called codewords.
The \emph{Hamming distance} between words (elements of the ambient space $\mathbb{F}_{q}^n$) $x=(x_1,\dots,x_n)$ and $y=(y_1,\dots,y_n)\in \mathbb{F}_q^n$, 
is the number \ $d(x,y)=\left| \{ i  :  x_i \neq y_i \} \right| $.
The \emph{minimum distance} of the code  $C$ is defined by \ $d=\mbox{min}\{ d(x,y):  x,y\in C, \ x\neq y\}$.  The \emph{weight} of a codeword $x$ is \ $w(x)=d(x,0)=|\{i  :  x_i\neq 0 \}|$. For a linear code, $d=\mbox{min}\{ w(x)  :  x \in C, x\neq0 \}.$ A  $q$-ary linear code of length $n$, dimension $k$, and distance $d$ is called an $[n,k,d]_q$ code. 

The \emph{dual} code $C^\perp$ of a code $C$ is the orthogonal complement of $C$ under the Euclidean inner product
$\langle\cdot \, ,\cdot\rangle$, i.e.\ $C^{\perp} = \{ v \in \mathbb{F}_{q}^n| \langle v,c\rangle=0 {\rm \ for\ all\ } c \in C \}$. A linear code $C$ over $\mathbb F_q$ is called an (Euclidean) LCD code if $C\cap C^{\perp}=\{0\}$. It follows easily that the dual of an LCD code is an LCD code.

A {\it subspace code} $C_S$ is a nonempty set of subspaces of $\mathbb F^n_q$. 
For the parameters of a subspace code we will follow the notation from \cite{table_sc}
and use a {\it subspace distance} given by
\begin{equation} \label{submet}
d_s(U,W)=\dim(U+W)-\dim(U\cap W),
\end{equation}
where $U, W \in C_S$. The {\it minimum distance} of $C_S$ is given by 
$$d=min \{d_S(U,W) | \ U,W \in C_S, U \neq W \}.$$
A code $C_S$ is called an $(n,\#C_S,d;K)_q$ subspace code if the dimensions of the codewords of $C_S$ are contained in a set $K\subseteq \{0,1,2, \dots, n\}$. 
In the case $K=\{k\}$, a subspace code $C_S$ is called a {\it constant dimension code} with the parameters $(n,\#C_S,d;k)_q$, otherwise, \textit{i.e.} if all codewords do not have the same dimension, 
$C_S$ is called a {\it mixed dimension code}. Such a subspace code is denoted by $(n,\#C_S,d)_q$. Constant dimension codes are the most studied subspace codes as they are $q$-analogues of the constant-weight codes. 

We will follow the definition of an association scheme given in \cite{BCN}, although some authors use a term a {\it symmetric association scheme} for such structures.

Let $X$ be a finite set. An {\it association scheme} with $d$ classes is a pair $(X,\mathcal{R})$ such that
\begin{enumerate}
 \item $\mathcal{R}= \{ R_0,R_1,\dots,R_d\}$ is a partition of $X \times X$,
 \item $R_0=\bigtriangleup = \{ (x,x) |x \in X \} $,
 \item $R_i=R_i^{\top}$ (\textit{i.e.} $(x,y) \in R_i \Rightarrow (y,x) \in R_i)$ for all $i \in \{ 0,1,\dots, d \} $,
 \item there are numbers $p_{ij}^k$ (the intersection numbers of the scheme) such that for any pair $(x,y) \in R_k$ the number of $z \in X$ such that $(x,z) \in R_i$ and $(z,y) \in R_j$ equals $p_{ij}^k$. 
\end{enumerate}

The relations $R_i$, $i \in \{ 0,1, \dots, d\}$, of an association scheme can be described by the set of symmetric $(0,1)$-adjacency matrices $\mathcal{A}= \{A_0, A_1,\dots, A_d \}$, $A_i=[a^i_{x,y}]$ 
for $i=0,1, \ldots, d$, where $a^i_{xy}=1$ if $(x,y) \in R_i$, which generate $(d+1)$-dimensional commutative and associative algebra over real or complex numbers called {\it the Bose-Mesner algebra} of the scheme. The matrices $\{A_0, A_1, \dots, A_d \}$ satisfy
\begin{equation} \label{form1}
 A_i A_j=\sum_{k=0}^d p_{i,j}^k A_k=A_j A_i.
\end{equation}
Each of the matrices $A_i$, $i \in \{1,2, \dots, d\}$, represents a simple graph $\Gamma_i$ on the set of vertices $X$ (if $(x,y)\in R_i$ then vertices $x$ and $y$ are adjacent in $\Gamma_i$). 

\section{LCD subspace codes }\label{LCD-subspace}

As an analog of the definition of an LCD linear code we introduce the definition of an LCD subspace code as follows.

\begin{defi}
Let ${\cal P}_q(n)$ be the set of all subspaces of $\mathbb F_q^n$, and let $C_S \subseteq  {\cal P}_q(n)$ be a subspace code. 
If $C_i \cap C_j^\perp=\{0\}$, for all $C_i,C_j \in C_S$, then $C_S$ is called an {\it LCD subspace code}. 
\end{defi}

Each codeword of an LCD subspace code $C_S$ is an LCD code, since for every $C_i \in C_S$ it holds that $C_i \cap C_i^\perp=\{0\}$.

The following lemma characterizes LCD codes, as shown in \cite{massey}.

\begin{lema}[Proposition 1, \cite{massey}]\label{LCD-herm}
Let $G$ be a generator matrix of a linear code over a finite field. Then $\mathrm{det}(GG^{\top})\neq 0$ if and only if $G$ generates an LCD code.
\end{lema}

For a construction of LCD subspace codes we will use the following property.

\begin{lema}\label{LCD-lemma}
Let $C_1$ and $C_2$ be $[n,k]$ codes over the field $\mathbb F_q$, and let $G_1$ and $G_2$ be their generator matrices, respectively. If the matrix $G_2G_1^{\top}$ is nonsingular,
then $C_2 \cap C_1^{\perp}=C_1 \cap C_2^{\perp}= \{0\}$.
\end{lema}
\begin{proof}
Suppose $v \in C_2$. Then there exists $u \in \mathbb F_q^{k}$ such that $v=uG_2$. If $v \in C_1^{\perp}$, then $vG_1^{\top}=0$, i.e
$uG_2G_1^{\top}=0$. Since $G_2G_1^{\top}$ is nonsingular, $u=0$ and therefore $v=0$. Hence, $C_2 \cap C_1^{\perp}= \{0\}$.

On the other hand, if $G_2G_1^{\top}$ is nonsingular then $(G_2G_1^{\top})^{\top}=G_1G_2^{\top}$ is also nonsingular. Therefore, $C_1 \cap C_2^{\perp}= \{0\}$.
\end{proof}

\bigskip

In \cite{massey}, Massey gave a decoding method for LCD codes. Below we propose a decoding algorithm for LCD subspace codes. In that purpose, we need the following.

Let $X \subseteq V$ and $y\in V$ such that $X \cap X^\perp=\{0\}$. 
Then it is easy to see that 
\begin{equation*} \label{dim}
\dim(X+\left\langle y\right\rangle)=\dim(X)+\dim(\pi_{X^\perp}(\left\langle y\right\rangle)), 
\end{equation*}
where $\pi_{X^\perp}(\left\langle y\right\rangle)$ is the projection on the space $X^\perp$. In general, if $Y=\left\langle y_1,\dots,y_k\right\rangle$ then $\pi_{X^\perp}(\left\langle y_1,\dots,y_k\right\rangle)=\left\langle \pi_{X^\perp}(y_1), \dots, \pi_{X^\perp}(y_k ) \right\rangle.$ Thus we can conclude that if $X, Y \in V$ and $X \cap X^\perp=\{0\}$, then
\begin{equation*} \label{dim1}
\dim(X+Y)=\dim(X)+\dim(\pi_{X^\perp}(Y)). 
\end{equation*}

That gives us the following lemma:

\begin{lema}
Let $\{C_1, \dots C_m\}$ be an LCD subspace code and let $C \leq \mathbb F^n_q$. Then,  
\begin{equation} \label{dim2}
d_s(C_i,C)=\dim(C_i)+2\dim(\pi_{C_i^\perp}(C))-\dim(C). 
\end{equation}
\end{lema}

\begin{proof}
Since
$$\dim(C_i+C)=\dim(C_i)+\dim(C)-\dim(C_i\cap C)$$ and 
$$\dim(C_i+C)=\dim(C_i)+\dim(\pi_{C_i^\perp}(C)),$$ the following formula holds
$$\dim(C_i\cap C)=\dim(C)-\dim(\pi_{C_i^\perp}(C)).$$
Since $$d_s(C_i,C)=\dim(C_i+C)-\dim(C_i\cap C),$$ the formula (\ref{dim2}) holds. 

\end{proof}

\begin{remark}
Let $C_S=\{C_1, \dots C_m\}$ be an LCD subspace code. Suppose that the codeword $C_i= \langle x_1^i,x_2^i, \ldots , x_k^i \rangle$ is sent through a noisy channel, and the subspace $C= \langle x_1,x_2, \ldots , x_k \rangle$ is received. Based on equation (\ref{dim2}), one can calculate the distance between $C$ and $C_i \in C_S$ in the following way
\begin{equation} \label{dim3}
d_s(C_i,C)=\dim(C_i)+2\dim(\langle \pi_{C_i^\perp}(x_1), \pi_{C_i^\perp}(x_2), \ldots, \pi_{C_i^\perp}(x_k)  \rangle)-\dim(C). 
\end{equation}
A comparison of the equations (\ref{dim3}) and (\ref{submet}) suggests that the minimum distance decoding problem for an LCD subspace code can be simpler than that for a general subspace code, which is similar to the case of classical LCD codes \cite{massey}. Note that a minimum distance decoder for subspace codes chooses the closest codeword to the received word with respect to the subspace distance. If there is more than one closest codeword, the decoder returns "failure".
\end{remark}

\section{LCD subspace codes from association schemes}\label{constr-LCD}
 
Let us consider a square $n \times n$ real matrix $A$ whose rows and columns are indexed by elements of $X = \{1, 2, \ldots , n \}$. Let $\Pi = \{X_1, X_2, \ldots, X_t \}$ be a partition of $X$. 
The characteristic matrix $H$ corresponding to $\Pi$ is the $n \times t$ matrix whose $j$th column is the characteristic vector of $X_j$, where $j= 1,\dots,t$.
We partition the matrix $A$ according to $\Pi$ as $A=[A_{ij}]$, $1 \le i,j \le t$. If $q_{ij}$ denotes the average row sum of $A_{ij}$ then the matrix 
$Q = [q_{ij}]$ is called a \emph{quotient matrix} of $A$. If the row sum of each block $A_{ij}$ is a constant then the partition $\Pi$ is called \emph{row equitable}. 
Similarly, if the column sum of each block $A_{ij}$ is a constant then the partition $\Pi$ is called \emph{column equitable}. If $\Pi$ is both row and column equitable, then $\Pi$ is said to be
\emph{equitable}. If $A$ is an adjacency matrix of a graph $\Gamma$ and $\Pi$ is an equitable partition of $A$, then we say that $\Pi$ is an equitable partition of the graph $\Gamma$.
An equitable (or regular) partition of an association scheme $(X,\mathcal{R})$ with $d$ classes is a partition of $X$ which is equitable with respect to each adjacency matrix of the graphs 
$\Gamma_i$, $i \in \{1,2, \dots,d \}$, corresponding to the association scheme $(X,\mathcal{R})$.

Let $\Pi$ be an equitable partition of an association scheme $(X,\mathcal{R})$ with $t$ cells, and let $H$ be the characteristic matrix of the partition $\Pi$.
Further, let $A_i$ be the adjacency matrix corresponding to a relation $R_i$.
Then the following holds:
\begin{equation} \label{form2}
 A_iH=HM_i,
\end{equation}
where $M_i$ denote the corresponding $t \times t$ quotient matrix of $A_i$ with respect to $\Pi$.
The matrix $H^{\top} H$ is diagonal and invertible and, therefore,
\begin{equation} \label{form2a}
M_i=(H^{\top} H)^{-1} H^{\top} A_iH.
\end{equation}

We will use the next result that was given in \cite{dsa-as}.

\begin{tm} \label{th_struc_const}
Let $\Pi$ be an equitable partition of a $d$-class association scheme $(X,\mathcal{R})$ with $t$ cells, 
and let $M_i$, $i=0, 1, \dots, d$, denote the quotient matrix of the graph $\Gamma_i$ with respect to $\Pi$. Then
\begin{equation} \label{form4}  
 M_i M_j=\sum_{k=0}^d p_{i,j}^k M_k=M_j M_i,
\end{equation}
where integers $p_{ij}^k$ are the intersection numbers of the scheme.
\end{tm}

The following theorem is given in \cite{as}. Note that $I_{t}$ denotes the identity matrix of order $t$.

\begin{tm} \label{lcd-code}
Let $\Pi$ be an equitable partition of a $d$-class association scheme $(X,\mathcal{R})$ with $t$ cells of the same length $\frac{|X|}{t}$. 
Further, let $M_i$ denotes the corresponding quotient matrix of $A_i$ with respect to $\Pi$. If there exists $i \in \{1,2,\dots,d\}$ and a prime $p$ such that for all $k \in \{0,1,\dots,d\}$ the prime $p$ divides 
$p_{i,i}^k$, then the matrix
$N=\left[ \begin{array}{c|c}
M_i & \alpha I_t
\end{array} \right]$
generates a $[2t,t]_{q}$ LCD code for every $\alpha \in \mathbb F_{q} \setminus \{ 0 \}$, for some positive integer $r$ after $q = p^r$.
\end{tm}

Let $\mathcal{A}= \{A_0, A_1, \dots, A_d \}$ be the set of adjacency matrices of the association scheme $(X,\mathcal{R})$, and let $q=p^r$ be a prime power.
Further, let $\Pi$ be an equitable partition of a $d$-class association scheme $(X,\mathcal{R})$ with $t$ cells of the same length $\frac{|X|}{t}$
and let $M_i$ denote the corresponding quotient matrix of $A_i$ with respect to $\Pi$.
Let us consider the matrices  
$N_i=\left[ \begin{array}{c|c}
M_i & \alpha_i I_t
\end{array} \right]$,
$i \in I=\{i_1, i_2, \dots, i_s\} \subseteq \{0, 1, \dots, d \}$, $s \geq 2$, and $\alpha_i \in \mathbb F_{q} \setminus \{ 0 \}$ for all $i \in I$.
Since the partition $\Pi$ is equitable  
having all cells of the same size, and the matrices $A_i$, $i=0, \ldots , d$, are symmetric, the matrices $M_i$, $i=0, \ldots , d$, are symmetric too.
If $p|p_{x,y}^k$, for all $k \in \{0,1,\dots,d\}$, where $x, y \in I$, then
$$M_xM_y^\top=M_i M_i=\sum_{k=0}^d p_{x,y}^k M_k=0,$$
and 
$$N_xN_y^\top=\alpha_x \alpha_y I_t.$$
Note that the above matrix equations are over $\mathbb F_q$. Especially, the matrix $N_xN_y^\top$ is nonsingular.

Let us now consider the matrix algebra $\overline {\cal M}$ over $\mathbb F_q$ generated by matrices $M_i$, $i \in I$, where $p|p_{x,y}^k$, for all $k \in \{0,1,\dots,d\}$, and all of $x, y \in I$. 
Then for nonzero matrices $X,Y \in \overline {\cal M}$ it holds that  
$\left[ \begin{array}{c|c}
X & \alpha_x I_t
\end{array} \right]
\left[ \begin{array}{c|c}
Y & \alpha_y I_t
\end{array} \right]^\top$is nonsingular, where $\alpha_x, \alpha_y \in \mathbb F_{q} \setminus \{ 0 \}$. 
The following theorem holds.

\begin{tm} \label{lcd_subcode}
Let $\Pi$ be an equitable partition of a $d$-class association scheme $(X,\mathcal{R})$ with $t$ cells of the same length $\frac{|X|}{t}$,
$\mathcal{A}= \{A_0, A_1, \dots, A_d \}$ be the set of adjacency matrices of $(X,\mathcal{R})$,
and let $M_i$ denote the corresponding quotient matrix of $A_i$ with respect to $\Pi$.
Further, let $I=\{i_1, i_2, \dots, i_s \} \subseteq \{0, 1, \dots, d \}$ and $p|p_{i,j}^k$, for all $k \in \{0,1,\dots,d\}$ and all $i, j \in I$, where $p$ is a prime number.
Then the set of row spaces of the matrices  
$N_x=\left[ \begin{array}{c|c}
X & \alpha_x I_t
\end{array} \right]$, $\alpha_x \in \mathbb F_{q} \setminus \{ 0 \}$,  
where $X$ is a nonzero element of the matrix algebra generated by the matrices $M_i$, $i \in I$, forms an LCD subspace code $C_S \subseteq \mathbb F_q^{2t}$, for some positive integer $r$ after $q = p^r$.
\end{tm}
\begin{proof}
For two matrices $N_x$ and $N_y$, the matrix $N_xN_y^\top=\alpha_x \alpha_y I_t$ is nonsingular. Lemma \ref{LCD-lemma} leads us to the conclusion that $C_S$ is an LCD subspace code.
\end{proof}

\begin{remark}
Note that the LCD subspace code $C_S$ constructed by applying Theorem \ref{lcd_subcode} is a constant dimension code with parameters $(2t,\#C_S,d;t)_q$.
\end{remark}

%

\subsection{Examples from distance-regular graphs}
In this subsection, we show how one can apply the method introduced in Theorem \ref{lcd_subcode} and obtain results.

Let $\Gamma$ be a graph with diameter $d$, and let $\delta(u,v)$ denote the distance between vertices $u$ and $v$ of $\Gamma$.
The $i$th-neighborhood of a vertex $v$ is the set $\Gamma_{i}(v) = \{ w : \delta(v,w) = i \}$. 
Similarly, we define $\Gamma_{i}$ to be the $i$th-distance graph of $\Gamma$, that is, the vertex set of $\Gamma_{i}$ is the same as for $\Gamma$, with adjacency in 
$\Gamma_{i}$ defined by the $i$th distance relation in $\Gamma$.
We say that $\Gamma$ is distance-regular if the distance relations of $\Gamma$ give the relations of a $d$-class association scheme, that is, for every choice of $0 \leq i,j,k \leq d$, 
all vertices $v$ and $w$ with $\delta(v,w)=k$ satisfy $|\Gamma_{i}(v) \cap \Gamma_{j}(w)| = p^{k}_{ij}$ for some constant $p^{k}_{ij}$.
In a distance-regular graph, we have that $p^{k}_{ij}=0$ whenever $i+j < k$ or $k<|i-j|$.
A distance-regular graph $\Gamma$ is necessarily regular with degree $p^{0}_{11}$; more generally, each distance graph $\Gamma_{i}$ is regular with degree $k_{i}=p^{0}_{ii}$.
An equivalent definition of distance-regular graphs is the existence of the constants $b_{i}=p^{i}_{i+1,1}$ and $c_{i}= p^{i}_{i-1,1}$ for $0 \leq i \leq d$ (notice that $b_{d}=c_{0}=0$).
The sequence $\{b_0,b_1,\dots,b_{d-1};c_1,c_2,\dots,c_d\}$, where $d$ is the diameter of $\Gamma$, is called the intersection array of $\Gamma$. 
Clearly, $b_0=k$, $b_d=c_0=0$, $c_1=1$.

Let $\Gamma$ be a distance-regular graph with diameter $d$ and adjacency matrix $A$, and let $A_i$ denote the distance-$i$ matrix of $\Gamma$, $i=0, 1, \dots, d$. 
Further, let $G$ be an automorphism group of $\Gamma$. 
In the sequel, the quotient matrix of $A_i$, $i=1, 2, \dots, d$, with respect to the orbit partition induced by $G$ will be denoted by $M_i$. 

Since the adjacency matrices $A_i$, $i=1, 2, \dots, d$, of a distance-regular graph give an association scheme, the next corollary follows from Theorem \ref{lcd_subcode}.

\begin{kor} \label{cor_drg_sub}
Let $\Gamma$ be a distance-regular graph with diameter $d$, and  let an automorphism group $G$ act on $\Gamma$ with $t$ orbits of the same length.
Further, let $I=\{ i_1, i_2, \dots, i_s \} \subseteq \{0, 1,\dots, d \}$ and $p$ be a prime number such that $p|p_{i,j}^k$, for all $k \in \{0,1,\dots,d\}$ and all $i, j \in I$.
Then the set of row spaces of the matrices  
$N_x=\left[ \begin{array}{c|c}
X & \alpha_x I_t
\end{array} \right]$, $\alpha_x \in \mathbb F_{q} \setminus \{ 0 \}$,  
where $X$ is a nonzero element of the matrix algebra generated by the matrices $M_i$, $i \in I$, forms an LCD subspace code $C_S \subseteq \mathbb F_q^{2t}$, for some positive integer $r$ after $q = p^r$.
\end{kor}

To illustrate the method given in Corollary \ref{cor_drg_sub}, we give the following examples.

\begin{ex}
Let $\Gamma_1$ be a distance-regular graph having 200 vertices, diameter $d=5$, and intersection array $\{22,21,16,6,1;1,6,16,21,22\}$ known as Doubled Higman-Sims graph, see \cite{BCN}. It is easy to see that a construction for obtaining LCD subspace codes described in Corollary \ref{cor_drg_sub} can be applied only to the quotient matrices $M_1$ and $M_4$ obtained with respect to the orbit partition induced by the action of an automorphism group of the graph $\Gamma_1$. Let $H_1\cong Z_5:Z_4$ be a subgroup of the full automorphism group of $\Gamma_1$. Applying Corollary \ref{cor_drg_sub} and the group $H_1$ one obtains an LCD subspace code with parameters $(20,16,2;10)_2$. Let $H_2\cong Z_{10}$ be a subgroup of the full automorphism group of $\Gamma_1$. Applying Corollary \ref{cor_drg_sub} and the group $H_2$ one obtains an LCD subspace code with parameters $(40,5,2;20)_2$.
\end{ex}

\begin{ex}
Let $\Gamma_2$ be a distance-regular graph having 154 vertices, diameter $d=5$, and intersection array $\{16,15,12,4,1;1,4,12,15,16\}$ known as Doubled $M_{22}$ graph, see \cite{BCN}. A construction described in Corollary \ref{cor_drg_sub} can be applied only to the quotient matrices $M_1$ and $M_4$ obtained with respect to the orbit partition induced by the action of an automorphism group of the graph $\Gamma_2$. Let $H\cong Z_{14}$ be a subgroup of the full automorphism group of $\Gamma_2$. Applying Corollary \ref{cor_drg_sub} and the group $H$ one constructs an LCD subspace code with parameters $(22,4,2;11)_2$.
\end{ex}

\section{Constructions from mutually unbiased Hadamard matrices} \label{MUH}

A \emph{Hadamard matrix} of order $n$ is a $n \times n$ $(-1, 1)$-matrix $H$ such that $HH^{\top} = n I_{n}$. 
It is well known that a Hadamard matrix of order $n$ can exists only if $n$ = 1, 2 or $n \equiv 0 \mod 4$. The Hadamard conjecture states
that these necessary conditions are also sufficient. Since the discovery of a Hadamard matrix of order 428, the smallest open case is $n$ = 668 (see \cite{Hadi-428}).

Two Hadamard matrices $H$ and $K$ of order $n$ are called \emph{unbiased} if $HK^{\top} = \sqrt{n}L$, where $L$ is a Hadamard matrix of order $n$.
Clearly, unbiased Hadamard matrices exist only in square orders.
If $\{ H_1, H_2, \ldots , H_m \}$ is a set of mutually unbiased Hadamard matrices of order $2n$, then $m \le n$ (see \cite[Theorem 2.]{MUH-CM}).
This upper bound is attained for Hadamard matrices of order $4^k$ (see \cite{qf-GF(2)}).
We refer the reader to \cite{MUH-CM, MUH-NATO} for more information on mutually unbiased Hadamard matrices.

\begin{tm} \label{lcd-MUH}
Let $\{ H_1, H_2, \ldots , H_m \}$ be a set of mutually unbiased Hadamard matrices of order $n$. Further, let $p$ be a prime number dividing $\sqrt{n}$ and $\mathbb{F}_q$ be the finite field of order $q$, for some positive integer $r$ after $q = p^r$. Then the set of row spaces of the matrices  
$N_x=\left[ \begin{array}{c|c}
X & \alpha_x I_n
\end{array} \right]$, $\alpha_x \in \mathbb F_{q} \setminus \{ 0 \}$,  
where $X$ is a nonzero element of the matrix algebra generated by the matrices $H_i$, $i=1,2, \ldots , m$, forms an LCD subspace code $C_S \subseteq \mathbb F_q^{2n}$.
\end{tm}

\begin{proof}
For matrices $H_i$ and $H_j$, $1 \le i,j \le m$, $i \neq j$, it holds that $H_iH_i^{\top}=H_jH_j^{\top}=nI_n$ and $H_iH_j^{\top}=\sqrt{n}L$, where $L$ is a Hadamard matrix.
Hence, for two matrices $N_x$ and $N_y$, the matrix $N_xN_y^\top=\alpha_x \alpha_y I_n$, which is a nonsingular matrix. It follows from Lemma \ref{LCD-lemma} that the subspace determined by $N_x$ intersect the orthogonal complement of the subspace determined by $N_y$ trivially, so $C_S$ is an LCD subspace code.
\end{proof}

Weighing matrices are a generalization of Hadamard matrices. A matrix $W = [w_{ij}]$ of order $n$ and $w_{ij} \in \{-1, 0, 1\}$ is called a weighing matrix with weight $k$, if $WW^{\top} = kI_n$.
If $k=n$, then $WW^{\top} = nI_n$ and the weighing matrix $W$ is a Hadamard matrix.

Two weighing matrices $W_1$ and $W_2$ of order $n$ and weight $k$ are called unbiased, if $W_1W_2^{\top}  = \sqrt{k}W$, where $W$ is a weighing matrix of order $n$ and weight $k$.
Mutually unbiased weighing matrices have been studied in \cite{MUH-CM, MUH-NATO}.
The following theorem is a generalization of Theorem \ref{lcd-MUH}, and can be proven in a similar way.

\begin{tm} \label{lcd-WUH}
Let $\{ W_1, W_2, \ldots , W_m \}$ be a set of mutually unbiased weighing matrices of order $n$ and weight $k$. Further, let $p$ be a prime number dividing $\sqrt{k}$ and $\mathbb{F}_q$ be the finite field of order $q$, for some positive integer $r$ after $q = p^r$. Then the set of row spaces of the matrices  
$N_x=\left[ \begin{array}{c|c}
X & \alpha_x I_n
\end{array} \right]$, $\alpha_x \in \mathbb F_{q} \setminus \{ 0 \}$,  
where $X$ is a nonzero element of the matrix algebra generated by the matrices $W_i$, $i=1,2, \ldots , m$, forms an LCD subspace code $C_S \subseteq \mathbb F_q^{2n}$.
\end{tm}

In order to illustrate the method given in Theorem \ref{lcd-WUH}, we give the following example.

\begin{ex}
Let $\{W_1, W_2, W_3, K\}$ be mutually unbiased weighing matrices of order 16 and weight nine from \cite[Example 19.]{MUH-CM}. The construction described in Theorem \ref{lcd-WUH} can be applied only for a construction of LCD subspace codes over the finite field of order three since the weights of the matrices are 9. Applying the method given in Theorem \ref{lcd-WUH}, the set of row spaces of the matrices 
$N_x$, where $X \in \{W_1, W_2, W_3, K \}$ and $\alpha_x\in \{1,2\}$, leads to a LCD subspace code $(32,81,6;16)_3$. However, one can consider all the subalgebras generated by the subsets of the set 
$\{W_1, W_2, W_3, K\}$ containing three matrices. Applying the method in these cases, we obtain the following: the set of row spaces of the matrices $N_x$, where $X \in \{W_1, W_2, K\}$ (and the same is obtained for the subsets $\{W_1, W_3, K\}$ and $\{W_2, W_3, K\}$) and $\alpha_x\in \{1,2\}$, leads to a LCD subspace code $(32,27,8;16)_3$, and the set of row spaces of the matrices $N_x$, where 
$X \in \{W_1, W_2, W_3\}$ and $\alpha_x\in \{1,2\}$, leads to a LCD subspace code $(32,27,6;16)_3$. This example shows that applying the method given in Theorem \ref{lcd-WUH} we can obtain interesting constant dimension LCD subspace codes.
\end{ex}

In the sequel, we will present construction of LCD subspace codes from quotient matrices and association schemes of three, five and eight classes obtained from mutually unbiased Hadamard matrices.

\subsection{Quotient matrices of mutually unbiased Hadamard matrices}

Let a group $G$ acts as an automorphism group of a Hadamard matrix $H$. Then $G$-orbits form an equitable partition of the matrix $H$ and the corresponding quotient matrix is called an orbit matrix of $H$ (see \cite{orb-Had}).

Let $\{ H_1, H_2, \ldots , H_m \}$ be a set of mutually unbiased Hadamard matrices of order $n$. Further, let the rows and columns of the matrices $\{ H_1, H_2, \ldots , H_m \}$ are indexed by the elements of the set $X= \{1, \ldots ,n \}$. An equitable partition of the set $\{ H_1, H_2, \ldots , H_m \}$ of mutually unbiased Hadamard matrices is a partition $\Pi= \{C_1, C_2, \ldots , C_{t} \}$ 
of $X$ which is equitable with respect to each of the matrices $H_i$, $i=1,2, \ldots ,m$. Let $C$ be the characteristic matrix of $\Pi$. Let $M_i$ denotes the $t \times t$ quotient matrix of $H_i$ with respect to the partition $\Pi$. Then
$$H_iC=CM_i.$$
The matrix $C^{\top}C$ is diagonal and invertible and therefore,
$$M_i=(C^{\top}C)^{-1}C^{\top} H_i C.$$

Let $H_1$ and $H_2$ be unbiased Hadamard matrices of order $n$. Further, let $\Pi$ be an equitable partition of the set $\{ H_1, H_2 \}$ with $t$ cells of the same length $\frac{n}{t}$. 
It holds that 
$$M_1M_2^{\top}=(C^{\top}C)^{-1}C^{\top}H_1CC^{\top}H_2^{\top}C((C^{\top}C)^{-1})^{\top}.$$
Since $C^{\top}C$ is diagonal matrix, $(C^{\top}C)^{-1}$ is also diagonal, hence
$$M_1M_2^{\top}=(C^{\top}C)^{-1}C^{\top}H_1CC^{\top}H_2^{\top}C(C^{\top}C)^{-1}.$$

Since 
$$H_1(CC^{\top})=(CC^{\top})H_1=Q_1 \otimes J_{\frac{n}{t}},$$ 
where $Q_1$ is the quotient matrix of $H_1$ with respect to $\Pi$ and $H_1H_2^{\top}=\sqrt{n}L$, for some Hadamard matrix $L$, it holds that
$$M_1M_2^{\top}=(C^{\top}C)^{-1}C^{\top}CC^{\top}H_1H_2^{\top}C(C^{\top}C)^{-1}=(C^{\top}C)^{-1}C^{\top}CC^{\top}\sqrt{n}LC(C^{\top}C)^{-1}.$$

Hence, the following theorem holds.

\begin{tm} \label{lcd_Had-orb}
Let $S=\{ H_1,H_2, \ldots , H_m \}$ be a set of mutually unbiased Hadamard matrices of order $n$.
Further, let $\Pi$ be an equitable partition of the set $S$, with $t$ cells of the same length $\frac{n}{t}$, and let $M_i$ denote the corresponding quotient matrix of $H_i$ with respect to 
$\Pi$, $i=1,2, \ldots, m$. Let $p$ be a prime number dividing $\sqrt{n}$.
Then the set of row spaces of the matrices  
$N_x=\left[ \begin{array}{c|c}
X & \alpha_x I_t
\end{array} \right]$, $\alpha_x \in \mathbb F_{q} \setminus \{ 0 \}$,  
where $X$ is a nonzero element of the matrix algebra generated by the matrices $M_1,M_2, \ldots , M_m$, forms an LCD subspace code $C_S \subseteq \mathbb F_q^{2t}$, for some positive integer $r$ after $q = p^r$.
\end{tm}

Similarly, the following theorem holds.

\begin{tm} \label{lcd_WM-orb}
Let $S=\{ W_1,W_2, \ldots , W_m \}$ be a set of mutually unbiased weighing matrices of order $n$ and weight $k$.
Further, let $\Pi$ be an equitable partition of the set $S$, with $t$ cells of the same length $\frac{n}{t}$, and let $M_i$ denote the corresponding quotient matrix of $W_i$ with respect to 
$\Pi$, $i=1,2, \ldots, m$. Let $p$ be a prime number dividing $\sqrt{k}$.
Then the set of row spaces of the matrices  
$N_x=\left[ \begin{array}{c|c}
X & \alpha_x I_t
\end{array} \right]$, $\alpha_x \in \mathbb F_{q} \setminus \{ 0 \}$,  
where $X$ is a nonzero element of the matrix algebra generated by the matrices $M_1,M_2, \ldots , M_m$, forms an LCD subspace code $C_S \subseteq \mathbb F_q^{2t}$, for some positive integer $r$ after $q = p^r$.
\end{tm}

\subsection{Mutually unbiased regular Hadamard matrices} \label{MURH}

A Hadamard matrix of order $n$ for which the row sums and column sums are all the same, necessarily $\sqrt{n}$, is called \emph{regular}. 
It is clear that regular Hadamard matrices exist only in square orders. It is known that there is no pair of unbiased (row) regular Hadamard matrices of order $4n^2$, $n$ odd (see \cite{UCH}).
The following construction of a 3-class association scheme from mutually unbiased regular Hadamard matrices can be found in \cite{vanDam, MUH-BT}.

Let $\{ H_1,H_2, \ldots , H_m \}$ be a set of mutually unbiased regular Hadamard matrices of order $4n^2$, with $m \ge 2$. It follows that $n$ must be even. Further, let
$$M = \left [ I_{4n^2} \quad \frac{1}{2n}H_1 \quad \frac{1}{2n}H_2 \quad \ldots \quad \frac{1}{2n}H_m \right ]^{\top} 
\left [ I_{4n^2} \quad \frac{1}{2n}H_1^{\top} \quad \frac{1}{2n}H_2^{\top} \quad \ldots \quad \frac{1}{2n}H_m^{\top} \right ]$$
be the Gramian of the set of matrices $\{ I_{4n^2}, \frac{1}{2n}H_1, \frac{1}{2n}H_2, \ldots , \frac{1}{2n}H_m \}$. Define the matrix $B=2n(M-I_{4n^2})$. Then $B$ is a symmetric $(0,-1,1)$-matrix.
Let $B=B_1-B_2$, where $B_1$ and $B_2$ are  disjoint $(0,1)$-matrices and 
$J_{4n^2}$ is the $4n^2$ by $4n^2$ matrix of all 1 entries.
Then, $I_{4n^2(m+1)},B_1,B_2$ and $B_3=I_{m+1} \otimes J_{4n^2} - I_{4n^2(m+1)}$ form a 3-class association scheme. The intersection numbers can be seen in the following equalities:

\begin{align*}
&B^2_1 = (2n^2+n)mI_{4n^2(m+1)}+(n^2+\frac{3}{2}n)(m-1)B_1+(n^2+\frac{1}{2}n)(m-1)B_2+(n^2+n)B_3, \\
&B_2^2 = (2n^2-n)mI_{4n^2(m+1)}+(n^2-\frac{1}{2}n)(m-1)B_1+(n^2-\frac{3}{2}n)(m-1)B_2+(n^2-n)B_3, \\
&B_1B_2 = (n^2-\frac{1}{2}n)(m-1)B_1+(n^2+\frac{1}{2}n)(m-1)B_2+n^2mB_3, \\
&B_1B_3 = (2n^2+n-1)B_1+(2n^2+n)B_2, \\
&B_2B_3 = (2n^2-n)B_1+(2n^2-n-1)B_2.
\end{align*}

The next theorem follows from Theorem \ref{lcd_subcode} and the intersection numbers of the association scheme formed by $I_{4n^2(m+1)},B_1,B_2$, $B_3$.

\begin{tm} \label{lcd_MURH}
Let $\{ H_1,H_2, \ldots , H_m \}$ be a set of mutually unbiased regular Hadamard matrices of order $4n^2$, with $m \ge 2$, and the $(0,1)$-matrices $B_1, B_2$ and $B_3$ be defined as above.
Further, let $\Pi$ be an equitable partition of the $3$-class association scheme formed by $B_0=I_{4n^2(m+1)},B_1,B_2$, $B_3$, with $t$ cells of the same length $\frac{4n^2}{t}$,
and let $M_i$ denote the corresponding quotient matrix of $B_i$ with respect to $\Pi$, $i=0,1,2,3$. Let $p$ be a prime number dividing $\frac{n}{2}$.
Then the set of row spaces of the matrices  
$N_x=\left[ \begin{array}{c|c}
X & \alpha_x I_t
\end{array} \right]$, $\alpha_x \in \mathbb F_{q} \setminus \{ 0 \}$,  
where $X$ is a nonzero element of the matrix algebra generated by the matrices $M_1$ and $M_2$, forms an LCD subspace code $C_S \subseteq \mathbb F_q^{2t}$, for some positive integer $r$ after $q = p^r$.
\end{tm}

\subsection{Mutually unbiased Bush-type Hadamard matrices}

In \cite{MUH-BT}, the authors constructed association schemes of five and eight classes from mutually unbiased Bush-type Hadamard matrices. We will show that these association schemes can be used
for a construction of LCD subspace codes.

A \emph{Bush-type Hadamard matrix} of order $4n^2$ is a block matrix $H= [H_{ij}]$ with block size $2n$, $H_{ii}=J_{2n}$ and $H_{ij}J_{2n}=J_{2n}H_{ij}=0$, $i \neq j$, $1 \le i \le 2n$, $1 \le j \le 2n$,
where $J_{2n}$ is the $2n$ by $2n$ matrix of all 1 entries. Obviously, Bush-type Hadamard matrices are regular.

For odd values of $n$ there is no pair of unbiased Bush-type Hadamard matrices of order $4n^2$ (see \cite{UCH}), but there are unbiased Bush-type Hadamard matrices of order $4n^2$ for $n$ even 
(see \cite{MUH-CM}). The upper bound of the number of mutually unbiased Bush-type Hadamard matrices of order $4n^2$ is $2n-1$ (see \cite{MUH-BT}), whereas the upper bound of mutually unbiased Hadamard matrices of order $4n^2$ is $2n^2$ (see \cite{MUH-CM}). 

Let $\{ H_1,H_2, \ldots , H_m \}$ be a set of mutually unbiased Bush-type Hadamard matrices of order $4n^2$, $m \ge 2$. Since $m \ge 2$, $n$ must be even.
Let the matrices $B_1$ and $B_2$ be defined as in Section \ref{MURH}.
It is shown in \cite{MUH-BT} that the matrices

\begin{align*}
A_0 &=I_{4n^2(m+1)},\\
A_1&=I_{m+1} \otimes I_{2n} \otimes (J_{2n}-I_{2n}),\\
A_2&=I_{m+1} \otimes (J_{2n}-I_{2n}) \otimes J_{2n},\\
A_3&= (J_{m+1}-I_{m+1}) \otimes I_{2n} \otimes J_{2n},\\
A_4&=B_1-A_3,\\
A_5&=B_2,
\end{align*}

form a 5-class association scheme, and that

\begin{align*}
A_1^2&=(2n-1)A_0+(2n-2)A_1,\\
A_1A_2&= (2n-1)A_2,\\
A_1A_3&= (2n-1)A_3,\\
A_1A_4&= (n-1)A_4+nA_5,\\
A_1A_5&=nA_4+(n-1)A_5,\\
A_2^2&=2n(2n-1)A_0+2n(2n-1)A_1+2n(2n-2)A_2,\\
A_2A_3&=2n(A_4+A_5),\\
A_2A_4&=A_2A_5= (2n-1)nA_3+(2n-2)n(A_4+A_5),\\
A_3^2&=2mn(A_0+A_1)+2n(m-1)A_3,\\
A_3A_4&=A_3A_5=mnA_2+(m-1)n(A_4+A_5),\\
A_4^2&=A_5^2= (2n^2-n)mA_0+(n^2-n)m(A_1+A_2)+(n^2-\frac{n}{2})(m-1)(A_3+A_4)\\
&+(n^2-\frac{3n}{2})(m-1)A_5, \\
A_4A_5&=n^2mA_1+m(n^2-n)A_2+(n^2-\frac{n}{2})(m-1)A_3\\
&+(n^2-\frac{3n}{2})(m-1)A_4+(n^2-\frac{n}{2})(m-1)A_5.
\end{align*}

The following theorem holds, which can be verified by checking the corresponding intersection numbers of the 5-class association scheme given above. 

\begin{tm} \label{lcd_MUBH-5}
Let $\{ H_1,H_2, \ldots , H_m \}$ be a set of mutually unbiased Bush-type Hadamard matrices of order $4n^2$, with $m \ge 2$, and let $A_0,A_1,A_2,A_3,A_4,A_5$ denote the matrices defined above.
Further, let $\Pi$ be an equitable partition of the $5$-class association scheme formed by $A_0,A_1,A_2,A_3,A_4,A_5$, with $t$ cells of the same length $\frac{4n^2}{t}$,
and let $M_i$ denote the corresponding quotient matrix of $A_i$ with respect to $\Pi$, $i=0,1,2,3,4,5$. Let $p$ be a prime number dividing $\frac{n}{2}$.
Then the set of row spaces of the matrices  
$N_x=\left[ \begin{array}{c|c}
X & \alpha_x I_t
\end{array} \right]$, $\alpha_x \in \mathbb F_{q} \setminus \{ 0 \}$,  
where $X$ is a nonzero element of the matrix algebra generated by the matrices $M_2,M_3,M_4$ and $M_5$, forms an LCD subspace code $C_S \subseteq \mathbb F_q^{2t}$, for some positive integer $r$ after $q = p^r$.
\end{tm}

In \cite{MUH-BT}, the authors also gave a construction of an 8-class association scheme form mutually unbiased Bush-type Hadamard matrices. This association scheme is formed by the following matrices:
\begin{displaymath}  
\tilde{A_0}=   \left[
\begin{tabular}{cc}
 $A_0$ & $0$   \\
 $0$   & $A_0$ \\
\end{tabular}                 \right], \quad
\tilde{A_1}=   \left[
\begin{tabular}{cc}
$A_1$  & $0$   \\
 $0$   & $A_1$ \\
\end{tabular}                 \right], \quad
\tilde{A_2}=   \left[
\begin{tabular}{cc}
 $0$   & $A_1$ \\
 $A_1$ & $0$   \\
\end{tabular}                 \right], \quad
\tilde{A_3}=   \left[
\begin{tabular}{cc}
$A_2$  & $A_2$ \\
$A_2$  & $A_2$ \\
\end{tabular}                 \right],
\end{displaymath}
\begin{displaymath}  
\tilde{A_4}=   \left[
\begin{tabular}{cc}
 $A_3$ & $0$   \\
 $0$   & $A_3$ \\
\end{tabular}                 \right], \quad
\tilde{A_5}=   \left[
\begin{tabular}{cc}
 $0$   & $A_3$ \\
 $A_3$ & $0$   \\
\end{tabular}                 \right], \quad
\tilde{A_6}=   \left[
\begin{tabular}{cc}
$A_4$  & $A_5$ \\
$A_5$  & $A_4$ \\
\end{tabular}                 \right], \quad
\tilde{A_7}=   \left[
\begin{tabular}{cc}
$A_5$  & $A_4$ \\
$A_4$  & $A_5$ \\
\end{tabular}                 \right],
\end{displaymath}
$  \ 
\tilde{A_8}=   \left[
\begin{tabular}{cc}
 $0$   & $A_0$ \\
 $A_0$ & $0$   \\
\end{tabular}                 \right], 
\text{where} \ A_0,A_1,A_2,A_3,A_4,A_5 \ \text{denote the matrices defined above.}$ \\

\begin{tm} \label{lcd_MUBH-8}
Let $\{ H_1,H_2, \ldots , H_m \}$ be a set of mutually unbiased Bush-type Hadamard matrices of order $4n^2$, with $m \ge 2$, and let $\tilde{A_0}, \tilde{A_1}, \ldots ,\tilde{A_8}$ 
denote the matrices defined above. Further, let $\Pi$ be an equitable partition of the $8$-class association scheme formed by $\tilde{A_0}, \tilde{A_1}, \ldots ,\tilde{A_8}$, with $t$ cells of the same length $\frac{4n^2}{t}$, and let $M_i$ denote the corresponding quotient matrix of $A_i$ with respect to $\Pi$, $i=0,1, \ldots ,8$. Let $p$ be a prime number dividing $n$.
Then the set of row spaces of the matrices  
$N_x=\left[ \begin{array}{c|c}
X & \alpha_x I_t
\end{array} \right]$, $\alpha_x \in \mathbb F_{q} \setminus \{ 0 \}$,  
where $X$ is a nonzero element of the matrix algebra generated by the matrices $M_3, M_4, M_5, M_6$ and $M_7$ forms an LCD subspace code $C_S \subseteq \mathbb F_q^{2t}$, for some positive integer $r$ after $q = p^r$.
\end{tm}
\begin{proof}
It holds that
\begin{displaymath}  
\tilde{A_3}^2=  2 \left[
\begin{tabular}{cc}
 $A_2^2$ & $A_2^2$   \\
 $A_2^2$ & $A_2^2$   \\
\end{tabular}                 \right], \quad
\tilde{A_3} \tilde{A_4}= \tilde{A_3} \tilde{A_5}=  \left[
\begin{tabular}{cc}
$A_2 A_3$  & $A_2 A_3$ \\
$A_2 A_3$  & $A_2 A_3$ \\
\end{tabular}                 \right], 
\end{displaymath}
\begin{displaymath}  
\tilde{A_3} \tilde{A_6}= \tilde{A_3} \tilde{A_7}=  \left[
\begin{tabular}{cc}
 $A_2(A_4+A_5)$ & $A_2(A_4+A_5)$   \\
 $A_2(A_4+A_5)$ & $A_2(A_4+A_5)$   \\
\end{tabular}                 \right], \quad
\tilde{A_4}^2= \tilde{A_5}^2= \left[
\begin{tabular}{cc}
 $A_3^2$ & $0$     \\
 $0$     & $A_3^2$ \\
\end{tabular}                    \right], 
\end{displaymath}
\begin{displaymath}  
\tilde{A_4} \tilde{A_5}=  \left[
\begin{tabular}{cc}
$0$     & $A_3^2$ \\
$A_3^2$ & $0$     \\
\end{tabular}            \right], \quad
\tilde{A_4} \tilde{A_6}= \tilde{A_5} \tilde{A_7}= \left[
\begin{tabular}{cc}
$A_3 A_4$  & $A_3 A_5$ \\
$A_3 A_5$  & $A_3 A_4$ \\
\end{tabular}                         \right], 
\end{displaymath}
\begin{displaymath}  
\tilde{A_4} \tilde{A_7}= \tilde{A_5} \tilde{A_6}= \left[
\begin{tabular}{cc}
$A_3 A_5$  & $A_3 A_4$ \\
$A_3 A_4$  & $A_3 A_5$ \\
\end{tabular}                  \right], \quad
\tilde{A_6}^2= \tilde{A_7}^2= \left[
\begin{tabular}{cc}
$A_4^2+A_5^2$ & $2A_4A_5$     \\
$2A_4A_5$     & $A_4^2+A_5^2$ \\
\end{tabular}                      \right], 
\end{displaymath}
\begin{displaymath}  
\tilde{A_6} \tilde{A_7}=  \left[
\begin{tabular}{cc}
$2A_4A_5$     & $A_4^2+A_5^2$ \\
$A_4^2+A_5^2$ & $2A_4A_5$     \\
\end{tabular}                 \right]. 
\end{displaymath}

Hence,
\begin{align*}
& \tilde{A^2_3} = 4n((2n-1)\tilde{A_0}+(2n-1)\tilde{A_1}+(2n-1)\tilde{A_2}+(2n-2)\tilde{A_3}+(2n-1)\tilde{A_8}), \\
& \tilde{A_3} \tilde{A_4} = \tilde{A_3} \tilde{A_5}= 2n(\tilde{A_6} + \tilde{A_7}), \\
& \tilde{A_3} \tilde{A_6} = \tilde{A_3} \tilde{A_7}= 2n(2n-1)(\tilde{A_4} + \tilde{A_5})+2n(2n-2)(\tilde{A_6} + \tilde{A_7}), \\
& \tilde{A_4}^2= \tilde{A_5}^2 = 2mn(\tilde{A_0} + \tilde{A_1}) + 2n(m-1) \tilde{A_4}, \\
& \tilde{A_4} \tilde{A_5} = 2mn(\tilde{A_2} + \tilde{A_8}) + 2n(m-1) \tilde{A_5}, \\
& \tilde{A_4} \tilde{A_6}= \tilde{A_5} \tilde{A_7}= \tilde{A_4} \tilde{A_7}= \tilde{A_5} \tilde{A_6}= mn \tilde{A_3} + (m-1)n (\tilde{A_6} + \tilde{A_7}), \\ 
& \tilde{A_6}^2= \tilde{A_7}^2= 2(2n^2-n)m \tilde{A_0} + 2(n^2-n)m (\tilde{A_1} + \tilde{A_3}) + 2n^2m \tilde{A_2} \\
& \qquad  + (2n^2-n)(m-1)(\tilde{A_4} + \tilde{A_5}) + (2n^2-n)(m-1) \tilde{A_6} + (2n^2-3n)(m-1) \tilde{A_7}, \\
& \tilde{A_6} \tilde{A_7}= 2(2n^2-n)m \tilde{A_8} + 2(n^2-n)m (\tilde{A_2} + \tilde{A_3}) + 2n^2m \tilde{A_1} \\
& \qquad  + (2n^2-n)(m-1)(\tilde{A_4} + \tilde{A_5}) + (2n^2-n)(m-1) \tilde{A_7} + (2n^2-3n)(m-1) \tilde{A_6}.
\end{align*}

Theorem \ref{lcd_subcode} completes the proof.
\end{proof}

%
%
%
%

\section{Statements and Declarations}

\subsection{Funding}
This work has been fully supported by {\rm C}roatian Science Foundation under the project 5713.

\subsection{Availability of data and material}
Not applicable.


\end{document}